\documentclass[letterpaper, 10 pt, conference]{ieeeconf}
\usepackage{amsfonts}
\usepackage{graphicx}
\usepackage{geometry}
\usepackage{amstext}
\usepackage{amsmath}
\usepackage{latexsym}
\usepackage{cite}
\usepackage{mathptmx}
\usepackage{times}

\IEEEoverridecommandlockouts
\newtheorem{theorem}{Theorem}

\newtheorem{definition}[theorem]{Definition}
\newtheorem{example}[theorem]{Example}

\newtheorem{lemma}[theorem]{Lemma}

\begin{document}

\title{{\LARGE \textbf{Optimal State Estimation Synthesis over Unreliable
Network in Presence of Denial-of-Service Attack: an Operator Framework
Approach}}}
\author{Mohammad Naghnaeian \thanks{%
M. Naghnaeian is with the Mechanical Engineering Department, Clemson
University, Clemson, SC, USA \texttt{\small mnaghna@clemson.edu}}}
\maketitle

\begin{abstract}
In this paper, we consider the problem of state-estimation in the presence
of Denial-of-Service (DoS) attack. We formulate this problem as an state
estimation problem for a plant with switching measured outputs. In the
absence of attack, the state-estimator has access to all measured outputs,
however, in the presence of attack, only a subset of all measurements are
made available to the state-estimator. We seek to find an state-estimator
that results in the minimum estimation error for the worst-case attack
strategy. First, we parameterize the set of all state-estimators that result
in stable estimation error for the worst-case attack scenario. Then, we will
show that any state-estimator in this set can be written as a generalized
Luenberger observer with an appropriately defined observer-gain. This
observer-gain, in general, can be an operator and possibly unbounded as
opposed to the classical static observer-gain. Furthermore, we will show
that finding the optimal state-estimator that results in the minimum
estimation error can be cast as a convex program over the set of stable
factors of the observer operator-gain. This optimization in, in fact, linear programming and tractable.
\end{abstract}

\section{Introduction}

Modern cyber-physical systems (CPS) typically consist of many smaller
components that are spread over a large spatial domain. The performance of
the system, in whole, depends on the synergistic integration of
computational components such as control or estimation algorithm and
physical components such as actuators or sensors. Connectivity to the
outside world and the critical nature of CPS has made such systems hot
targets for adversarial attacks, see e.g. \cite{powner2007critical} and \cite%
{eisenhauer2006roadmap}. Denial-of-Service attack is an adversarial attack
in which the attacker disrupts the exchange of the information \cite%
{cardenas2008secure}. In the control theoretic context, the disrupted
information could be sensor measurements or control inputs to the actuators.
In this paper, we seek to design state-estimators that are resilient with
respect to the DoS attacks on the measurement channels. Such a problem has
been given some attention in the literature, e.g., in \cite{amin2009safe}, 
\cite{gupta2010optimal}, and \cite{li2015jamming}. Most of the existing
results aims at optimizing a cost-function, which is a measure of estimation
error, over a finite horizon in the stochastic/probabilistic framework where
a distribution form for the attacker or transmitter is assumed. In this
paper, however, we address this problem in the deterministic framework and
infinite horizon objective.

Our perspective is to think of a DoS attack as a switch and model the system
as a Linear Switching System (LSS). The attacker's strategy is to choose the
switching to maximize the estimation error and possibly destabilize that
while having a complete knowledge about system. On the other hand, the
state-estimator's strategy is to minimize the estimation error based on the
available sensor measurements as well as the current and past actions of the
attacker. We first, parametrize the set of \textbf{all} state-estimators
that result in bounded estimation error. We refer to such state-estimators
as stable estimators. Then we will define a new class of state observers
mimicking the conventional Luenberger observers. We will refer to this new
class as \textit{generalized Luenberger observers}. A generalized Luenberger
observer, in form, is very similar to a classical Luenberger observer with a
significant difference that its observer gain is an operator, and possibly
an unstable one, as opposed to a static gain in the classical observer. By
allowing the Luenberger observe to have an operator-gain, we will show that
the set of generalized Luenberger observers capture all stable
state-estimators. This, by itself, is a new result to the best of our
knowledge. Then, in order to find the optimal observer resulting in he
minimum estimation error, we formulate the problem as a convex optimization
over the stable factors of the observer operator-gains. These factors, and
the resulting observer operator-gain, are switching operators that causally
depend on the switching sequence (attacker's strategy). Finding the optimal
state-estimator, in fact, can be cast as a linear program and hence is
tractable.

Our approach relies on utilizing the operator framework which was first
introduced and developed in \cite{naghnaeian2016characterization} and \cite%
{naghnaeian2018l_p}. This operator framework provides a powerful tool to
study any type of linear system, time invariant, time varying, delayed,
switching, etc., in a unified way. Recently, the author has used such a
framework for the synthesis of decentralized controllers \cite%
{naghnaeian2018unified}. In what follows, we first review some results on
the switching systems and their operator representation and then present our
results on optimal state-estimator design subject to DoS.

\section{Preliminaries}

\subsection{Generic Notation}

We use $\mathbb{R}^{n}$ for the set of vectors of real numbers of dimension $%
n$. Given $x=\left\{ x\left( k\right) \right\} _{k=1}^{n}\in \mathbb{R}^{n}$%
, its $l_{\infty }$ norm is defined as $\left\Vert x\right\Vert =\max_{k\in
\left\{ 0,1,...,n-1\right\} }\left\vert x\left( k\right) \right\vert $. For
a (infinite dimensional) sequence $x=\left\{ x\left( k\right) \right\}
_{k=0}^{\infty }$ with $x\left( k\right) \in \mathbb{R}^{n}$, the $l_{\infty
}$ norm is defined by $\left\Vert x\right\Vert =\sup_{k}\left\Vert x\left(
k\right) \right\Vert $ whenever finite. The space of sequences with elements
in $\mathbb{R}^{n}$ whose $l_{\infty }$ norm is bounded is denoted by $%
l_{\infty }^{n}$. Throughout this paper, we view linear systems as mapping
on the space of $l_{\infty }^{n}$, for some positive integer $n$. In
general, for two normed spaces $\left( X,\left\Vert .\right\Vert _{X}\right) 
$ and $\left( Y,\left\Vert .\right\Vert _{Y}\right) $ and a linear operator $%
R:X\rightarrow Y$, the$\ $induced norm of this operator is given by $%
\left\Vert R\right\Vert _{X-Y}:=\sup_{x\neq 0}\frac{\left\Vert Rx\right\Vert
_{Y}}{\left\Vert x\right\Vert _{X}}$. The operator $R$ is said to be bounded
if its induced norm is finite. In this paper, we typically have $X=l_{\infty
}^{n}$ and $Y=l_{\infty }^{m}$, for some positive integers $n$ and $m$, and
we simply write $\left\Vert R\right\Vert $ to denote the $l_{\infty }^{n}$
to $l_{\infty }^{m}$ induced norm of the operator $R$. Any linear causal
operator $R$ can be thought of as an infinite dimensional lower triangular
matrix, 
\begin{equation}
R=\left[ 
\begin{array}{cccc}
R_{0,0} & 0 & 0 & \cdots \\ 
R_{1,1} & R_{1,0} & 0 & \cdots \\ 
R_{2,2} & R_{2,1} & R_{2,0} &  \\ 
\vdots & \vdots &  & \ddots%
\end{array}%
\right] .  \label{eq:matrix}
\end{equation}

\begin{definition}
A causal operator $R$ given by (\ref{eq:matrix}), is said to be bounded or
stable (on the space of $l_{\infty }$ sequences) if 
\begin{equation*}
\sup_{k}\left\Vert \left[ 
\begin{array}{cccc}
\cdots & R_{k,2} & R_{k,1} & R_{k,0}%
\end{array}%
\right] \right\Vert <\infty .
\end{equation*}
\end{definition}

Given a sequence $x=\left\{ x\left( k\right) \right\} _{k=0}^{\infty }$, the
delay or shift operator $\Lambda $ is defined by%
\begin{equation*}
\Lambda ^{k}x=\left\{ \underset{k\text{ zeros}}{\underbrace{0,...,0}}%
,x\left( 0\right) ,x\left( 1\right) ,...\right\} .
\end{equation*}

\begin{definition}
A linear causal map $R$ is called time-invariant if $\Lambda R=R\Lambda $.
\end{definition}

A Linear Time-Invariant (LTI) operator $R$ is fully characterize by its
impulse response denoted by $\left\{ R\left( k\right) \right\}
_{k=0}^{\infty }$ and its infinite dimensional matrix representation is
given by%
\begin{equation*}
R=\left[ 
\begin{array}{cccc}
R\left( 0\right) & 0 & 0 & \cdots \\ 
R\left( 1\right) & R\left( 0\right) & 0 & \cdots \\ 
R\left( 2\right) & R\left( 1\right) & R\left( 0\right) &  \\ 
\vdots & \vdots &  & \ddots%
\end{array}%
\right] .
\end{equation*}%
A Linear Time-Varying (LTV) system $R$ can also be written in state-space
representation as 
\begin{equation}
R:\left\{ 
\begin{array}{c}
x\left( t+1\right) =A\left( t\right) x\left( t\right) +B\left( t\right)
w\left( t\right) \\ 
y\left( t\right) =C\left( t\right) x\left( t\right) +D\left( t\right)
w\left( t\right)%
\end{array}%
\right. ,\text{ with }x\left( t_{0}\right) =x_{0},  \label{eq:01'}
\end{equation}%
where $u\left( t\right) \in \mathbb{R}^{m},x\left( t\right) \in \mathbb{R}%
^{n},$ $y\left( t\right) \in \mathbb{R}^{p}$, and $x_{0}\in \mathbb{R}^{n}$
are input, state, output, and the initial condition of the system and $%
A\left( .\right) $, $B\left( .\right) $, $C\left( .\right) $, and $D\left(
.\right) $ are matrices with appropriate dimensions for all $t$. Throughout
this paper, we think of linear systems as operators and hence we do not
directly work with the state-space representation. We, rather, convert the
state-space (\ref{eq:01'}) to (\ref{eq:matrix}). To do so, given a sequence
of matrices $\left\{ A\left( k\right) \right\} _{k=0}^{\infty }$, we define $%
\bar{A}$ to be the diagonal operator%
\begin{equation}
\bar{A}=\left[ 
\begin{array}{ccc}
A\left( 0\right) & 0 & \cdots \\ 
0 & A\left( 1\right) &  \\ 
\vdots &  & \ddots%
\end{array}%
\right] .  \label{eq:diag}
\end{equation}%
Using this notation, we can define diagonal operators $\bar{A}$, $\bar{B}$, $%
\bar{C}$, and $\bar{D}$ and rewrite (\ref{eq:01'}) as

\begin{equation}
R:\left\{ 
\begin{array}{c}
x=\Lambda \bar{A}x+\Lambda \bar{B}w+\bar{x}_{0} \\ 
y=\bar{C}x+\bar{D}w%
\end{array}%
\right. ,  \label{eq:01''}
\end{equation}%
where $\bar{x}_{0}=\left\{ \underset{t_{0}\text{ zeros}}{\underbrace{0,...,0}%
},x_{0},0,0,...\right\} $, $x=\left\{ x\left( t\right) \right\}
_{t=0}^{\infty }$, $y=\left\{ y\left( t\right) \right\} _{t=0}^{\infty }$, $%
w=\left\{ w\left( t\right) \right\} _{t=0}^{\infty }$, and $\Lambda $ is the
delay operator. The above representation of $R$ is referred to as the 
\textit{operator form}.

\begin{definition}
\label{def:stab}System $R$ in (\ref{eq:01''}) is said to be stable or
bounded if it is a bounded operator from $\left( 
\begin{array}{c}
\bar{x}_{0} \\ 
w%
\end{array}%
\right) $ to $\left( 
\begin{array}{c}
x \\ 
y%
\end{array}%
\right) $. More precisely, $R$ is stable if there exists a nonnegative real
number $\gamma \geq 0$ such that $\max \left\{ \left\Vert x\right\Vert
,\left\Vert y\right\Vert \right\} \leq \gamma \max \left\{ \left\Vert \bar{x}%
_{0}\right\Vert +\left\Vert w\right\Vert \right\} $ for all $\bar{x}%
_{0},w\in l_{\infty }$.
\end{definition}

\subsection{Linear Switched Systems}

In this section, we need to review some standard results on Linear Switched
Systems (LSS) presented in \cite{naghnaeian2016characterization} and \cite%
{naghnaeian2018l_p}. A Linear Switched System, $P_{\sigma }$, can be
represented in state-space by%
\begin{equation}
P_{\sigma }:\left\{ 
\begin{array}{c}
x\left( t+1\right) =A_{\sigma \left( t\right) }x\left( t\right) +B_{\sigma
\left( t\right) }u\left( t\right) \\ 
y\left( t\right) =C_{\sigma \left( t\right) }x\left( t\right) +D_{\sigma
\left( t\right) }u\left( t\right)%
\end{array}%
\right. ,  \label{eq:LSS}
\end{equation}%
where $\sigma =\left\{ \sigma _{k}\right\} _{k=0}^{\infty }$ is called the
switching sequence that takes values a finite set. Sometimes, $\sigma $ is
restricted to be in the set of admissible switching sequences $\Xi $. In the
operator framework, (\ref{eq:LSS}) can be written as%
\begin{equation}
P_{\sigma }:\left\{ 
\begin{array}{c}
x=\Lambda \bar{A}_{\sigma }x+\Lambda \bar{B}_{\sigma }u+\bar{x}_{0} \\ 
y=\bar{C}_{\sigma }x+\bar{D}_{\sigma }u%
\end{array}%
\right. ,  \label{eq:LSS1}
\end{equation}%
where $\bar{A}_{\sigma }=diag\left( A_{\sigma \left( 0\right) },A_{\sigma
\left( 1\right) },A_{\sigma \left( 2\right) },...\right) $ and $\bar{B}%
_{\sigma }$, $\bar{C}_{\sigma }$, and $\bar{D}_{\sigma }$ are defined
analogously.

There are important sub-classes of LSS that are of interest in this paper.
These are the LSS whose state matrices, A-matrices, remain constant and are
defined below:

\begin{definition}
We say a LSS $P_{\sigma }$ is an input-output LSS of degree $M$, for some
positive integer $M$, if it can be written, in state-space, as follows%
\begin{equation}
P_{\sigma }:\left\{ 
\begin{array}{c}
x\left( t+1\right) =A_{\sigma \left( t\right) }x\left( t\right) +B_{\sigma
\left( t\right) }u\left( t\right) \\ 
y\left( t\right) =C_{\left\{ \sigma \left( k\right) \right\}
_{k=t-M+1}^{t}}x\left( t\right) +D_{\left\{ \sigma \left( k\right) \right\}
_{k=t-M+1}^{t}}u\left( t\right)%
\end{array}%
\right. .  \label{eq:IOLSS}
\end{equation}%
We will denote the class of such systems by $\mathcal{S}_{IO}^{M}$ and $%
\mathcal{S}_{IO}=\bigcup\limits_{M=1}^{\infty }\mathcal{S}_{IO}^{M}$.
\end{definition}

We are also interested in a subclass of input-output LSS, output-only
switching, as follows:

\begin{definition}
A LSS $P_{\sigma }$ is said to be an output-only LSS of degree $M$ if it
admits the realization%
\begin{equation}
P_{\sigma }:\left\{ 
\begin{array}{c}
x\left( t+1\right) =A_{\sigma \left( t\right) }x\left( t\right) +Bu\left(
t\right) \\ 
y\left( t\right) =C_{\left\{ \sigma \left( k\right) \right\}
_{k=t-M+1}^{t}}x\left( t\right) +D_{\left\{ \sigma \left( k\right) \right\}
_{k=t-M+1}^{t}}u\left( t\right)%
\end{array}%
\right. .
\end{equation}%
The class of such systems is denoted by $\mathcal{S}_{O}^{M}$ and $\mathcal{S%
}_{O}=\bigcup\limits_{M=1}^{\infty }\mathcal{S}_{O}^{M}$.
\end{definition}

The classes of input-output and output-only LSS are rich classes since any
stable LSS can be approximated by elements of $\mathcal{S}_{O}$ and $%
\mathcal{S}_{IO}$ with arbitrary accuracy.

\begin{lemma}
\label{lem:02}Let $P_{\sigma }$ be a stable LSS and $\varepsilon >0$. Then,
there exist an integer $M$, $\bar{P}_{\sigma }\in \mathcal{S}_{IO}^{M}$, and 
$\tilde{P}_{\sigma }\in \mathcal{S}_{O}^{M}$ such that%
\begin{eqnarray*}
\left\Vert P_{\sigma }-\bar{P}_{\sigma }\right\Vert &<&\varepsilon , \\
\left\Vert P_{\sigma }-\tilde{P}_{\sigma }\right\Vert &<&\varepsilon ,
\end{eqnarray*}%
for any switching sequence $\sigma $. Moreover, $\bar{P}_{\sigma }$ and $%
\tilde{P}_{\sigma }$ can be made FIR (Finite-Impulse-Response).
\end{lemma}

Furthermore, there exist tractable and exact expressions to calculate the $%
l_{\infty }$ induced norm of LSS. In \cite{naghnaeian2018l_p}, it is proved
that the gain computation can be cast as a Linear Program. We do not review
those results here but rather refer the reader to \cite{naghnaeian2018l_p}%
.\bigskip

\section{Problem Setup}

Consider a linear plant given by%
\begin{eqnarray}
x &=&\Lambda \bar{A}x+\Lambda \bar{B}w+\bar{x}_{0},  \notag \\
y_{i} &=&\bar{C}_{i}x+\bar{D}_{i}w,  \label{eq:plant}
\end{eqnarray}%
where $x$ and $w$ are the states and exogenous disturbances, respectively,
and $y_{i}$'s, for $i=1,2,...,N$ for some integer $N$, are the
measurements/observations from this system. In this paper, we address the
problem of remote state-estimation where some of the measurements, $y_{i}$%
's, might not be available to the state-estimator due to intermittent
communication network or Denial-of-Service type of attack. In the ideal
nominal operating condition, when there is no DoS attack, the
state-estimator receives all $y_{i}$'s. That is, available information to
the state-estimator, $y_{a}$, is given by%
\begin{equation}
y_{a}^{0}=\left[ 
\begin{array}{c}
y_{1} \\ 
y_{2} \\ 
\vdots \\ 
y_{N}%
\end{array}%
\right] =\left[ 
\begin{array}{c}
\bar{C}_{1} \\ 
\bar{C}_{2} \\ 
\vdots \\ 
\bar{C}_{N}%
\end{array}%
\right] x+\left[ 
\begin{array}{c}
\bar{D}_{1} \\ 
\bar{D}_{2} \\ 
\vdots \\ 
\bar{D}_{N}%
\end{array}%
\right] w.  \label{eq:y0}
\end{equation}%
However, when a DoS attack occurs at the measurement channel, the
state-estimator only receives a subset of measurements. In this case,%
\begin{equation}
y_{a}^{\sigma \left( t\right) }\left( t\right) =E^{\sigma \left( t\right)
}y_{a}^{0}(t),  \label{eq:05}
\end{equation}%
where $E^{\sigma \left( t\right) }$ is a block diagonal matrix with identity
corresponding to $y_{i}$'s that are available to state-estimator and zero
otherwise. In the above expression, $\sigma \left( .\right) $ is the
switching signal orchestrating between modes of the system and take value in
some finite set. We use the zeroth mode to denote the nominal mode. For a
concrete example see below: 
\begin{figure}[t]
\centering
\par
\includegraphics[trim= 4.5in 4in 4in 0.5in,clip,width=\linewidth]{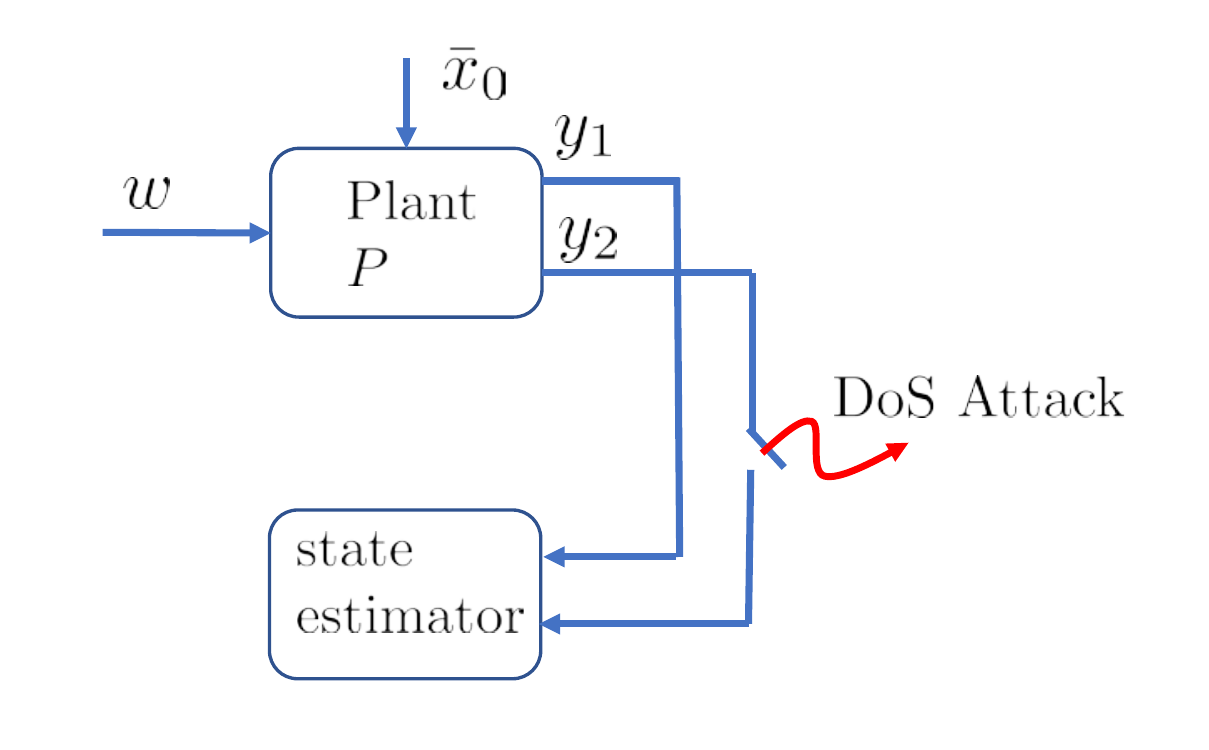}
\caption{DoS attack only on $y_{2}$ channel.}
\label{fig:1}
\end{figure}

\begin{example}
\label{exm:01}Consider a system in Figure \ref{fig:1} where there are two
measurements, $y_{1}$ and $y_{2}$. Suppose, the plant is \textbf{unstable}
LTI given by%
\begin{equation*}
x\left( t+1\right) =\left[ 
\begin{array}{ccc}
1 & 0 & 1 \\ 
-1 & 1 & 1 \\ 
-1 & 0 & 2%
\end{array}%
\right] x(t),x(0)=\left[ 
\begin{array}{c}
0.1 \\ 
0.2 \\ 
-0.1%
\end{array}%
\right] ,
\end{equation*}%
and the measurements are given by%
\begin{eqnarray*}
y_{1}\left( t\right) &=&\left[ 
\begin{array}{ccc}
0 & 1 & 0%
\end{array}%
\right] x(t)+2w_{1}\left( t\right) , \\
y_{2}\left( t\right) &=&\left[ 
\begin{array}{ccc}
1 & -1 & -2%
\end{array}%
\right] x(t)+0.01w_{2}(t),
\end{eqnarray*}%
where $\left\vert w_{i}\left( t\right) \right\vert \leq 1$, for $i=1,2$. In
this example, $y_{1}$ is a reliable measurement but with higher level of
disturbance and $y_{2}$ is an unreliable measurement with lower level of
disturbance. The DoS type of attack may result in measurement $y_{2}$ to not
reach the state-estimator. In this case, the available information, at each
time instant, to the state-estimator, $y_{a}\left( t\right) $ is given by (%
\ref{eq:05}), where 
\begin{equation*}
y_{a}^{0}\left( t\right) =\left[ 
\begin{array}{c}
y_{1}\left( t\right) \\ 
y_{2}\left( t\right)%
\end{array}%
\right] ,
\end{equation*}%
and%
\begin{equation*}
E^{\sigma \left( t\right) }\in \left\{ \left[ 
\begin{array}{cc}
I & 0 \\ 
0 & I%
\end{array}%
\right] ,\left[ 
\begin{array}{cc}
I & 0 \\ 
0 & 0%
\end{array}%
\right] \right\} .
\end{equation*}%
Therefore, $y_{a}^{\sigma \left( t\right) }=C^{\sigma \left( t\right)
}x+D^{\sigma \left( t\right) }w$ where%
\begin{eqnarray}
C^{\sigma \left( t\right) } &\in &\left\{ \left[ 
\begin{array}{ccc}
0 & 1 & 0 \\ 
1 & -1 & -2%
\end{array}%
\right] ,\left[ 
\begin{array}{ccc}
0 & 1 & 0 \\ 
0 & 0 & 0%
\end{array}%
\right] \right\} ,  \label{eq:C} \\
D^{\sigma \left( t\right) } &\in &\left\{ \left[ 
\begin{array}{cc}
2 & 0 \\ 
0 & 0.01%
\end{array}%
\right] ,\left[ 
\begin{array}{cc}
2 & 0 \\ 
0 & 0%
\end{array}%
\right] \right\} .  \label{eq:D}
\end{eqnarray}
\end{example}

Similarly to this example, we can rewrite (\ref{eq:05}) in the operator
framework and combine with (\ref{eq:plant})-(\ref{eq:y0}) to obtain the
following plant and attack model: 
\begin{eqnarray*}
x &=&\Lambda \bar{A}x+\Lambda \bar{B}w+\bar{x}_{0}, \\
y_{a}^{\sigma } &=&\bar{C}^{\sigma }x+\bar{D}^{\sigma }w,
\end{eqnarray*}%
where the switching sequence $\sigma \left( .\right) $ is the attacker's
strategy; we use $\sigma \left( t\right) =0$ to denote the nominal condition
at time instant $t$. In this expression, $y_{a}^{\sigma }$ is the sequence
of available information to state-estimator and $\sigma $ belongs to the set
of admissible sequences $\Xi $. In the above example, $\Xi $ is the set of
binary sequences.

\section{Main Results}

\subsection{Parametrization of State-Estimators}

In this section, we are interested to parametrize the set of
state-estimators. A state-estimator is a causal map, $T^{\sigma }$, from the
available measurements, $y_{a}^{\sigma }$, to a signal $\hat{x}$ which is
the estimation of state $x$. That is,%
\begin{equation}
\hat{x}=T^{\sigma }y_{a}^{\sigma }.  \label{eq:SE}
\end{equation}%
In the above expression, the dependency of the state-estimator on $\sigma $
is made explicit. We emphasize that $\sigma \left( .\right) $ is the
attacker's strategy which is causally known to the state-estimator. That is,
the state-estimator, at any given time, does not know the attacker's
intention in future but know its current and past actions. Therefore, $%
T^{\sigma }$ only causally depends on $\sigma $. In fact, a generic LSS as
given in (\ref{eq:LSS}) respects this causality. Henceforth, whenever an
operator's dependency on $\sigma $ is stated, causal dependency is assumed.

\begin{definition}
We say an state-estimator (\ref{eq:SE}) is stable if the estimation error $%
\tilde{x}:=\hat{x}-x$ is a bounded signal.\bigskip
\end{definition}

In the sequel, we first parametrize the set of all stable state-estimators
and then we will present our result on the synthesis of optimal
state-estimator that is resilient the DoS attacks.

\begin{lemma}
\label{lem:01}The set of all stable state-estimators (\ref{eq:SE}) that
result in a bounded estimation error is parametrized by bounded operators $%
T^{\sigma }$ and $X^{\sigma }$ such that%
\begin{equation}
\left[ 
\begin{array}{cc}
T^{\sigma } & X^{\sigma }%
\end{array}%
\right] \left[ 
\begin{array}{c}
\bar{C}^{\sigma } \\ 
\Lambda \bar{A}-I%
\end{array}%
\right] =I,\text{ for all }\sigma \in \Xi \text{.}  \label{eq:01}
\end{equation}%
In this case, the state-estimator and estimation error are given by (\ref%
{eq:SE}) and%
\begin{equation*}
\tilde{x}=X^{\sigma }\left( \Lambda Bw+\bar{x}_{0}\right) .
\end{equation*}
\end{lemma}

\begin{proof}
Let the state-estimator given by (\ref{eq:SE}). Then, the error is given by%
\begin{eqnarray*}
\tilde{x} &=&T^{\sigma }\left( \bar{C}^{\sigma }x+\bar{D}^{\sigma }w\right)
-x=\left( T^{\sigma }\bar{C}^{\sigma }-I\right) x+T^{\sigma }\bar{D}^{\sigma
}w \\
&=&\left( T^{\sigma }\bar{C}^{\sigma }-I\right) \left( I-\Lambda \bar{A}%
\right) ^{-1}\Lambda Bw \\
&&+T^{\sigma }\bar{D}^{\sigma }w+\left( T^{\sigma }\bar{C}^{\sigma
}-I\right) \left( I-\Lambda \bar{A}\right) ^{-1}\bar{x}_{0}.
\end{eqnarray*}%
Notice that $\tilde{x}$ is a bounded signal for bounded $w$, $\bar{x}_{0}$,
and $\sigma \in \Xi $ if and only if the mappings $T^{\sigma }$ and $\left(
T^{\sigma }\bar{C}^{\sigma }-I\right) \left( I-\Lambda \bar{A}\right) ^{-1}$
are bounded. Define%
\begin{equation*}
X^{\sigma }:=\left( T^{\sigma }\bar{C}^{\sigma }-I\right) \left( I-\Lambda 
\bar{A}\right) ^{-1}.
\end{equation*}%
Post-multiplying both sides by $\left( I-\Lambda \bar{A}\right) $, we obtain%
\begin{equation*}
X^{\sigma }\left( I-\Lambda \bar{A}\right) =\left( T^{\sigma }\bar{C}%
^{\sigma }-I\right) ,
\end{equation*}%
which is equivalent to (\ref{eq:01}) and this completes the proof.
\end{proof}

\bigskip

Traditionally, the state-estimation has been carried out utilizing
Luenberger observers. Luenberger observers, in their conventional shape,
form a strict subset of all stable state-estimators parametrized above. In
what follows, we introduce the \textit{Generalized Luenberger Observers }%
that differ from conventional ones in that their observer gains are
(possibly unstable) operators as opposed to static. A generalized Luenberger
is of the form%
\begin{equation}
\hat{x}=\Lambda \bar{A}\hat{x}+L^{\sigma }\left( \bar{C}^{\sigma }\hat{x}%
-y_{a}^{\sigma }\right) ,  \label{eq:GLO}
\end{equation}%
where $\hat{x}$ is the estimation of the state, $L$ is the observer
(possibly unbounded )operator-gain, and $y_{a}^{\sigma }$ is the available
information to state-estimator.

\begin{theorem}
\label{thm:01}Any stable state-estimator can be written as in (\ref{eq:GLO})
for an appropriate $L^{\sigma }$ in the form%
\begin{equation}
L^{\sigma}=\left( I+Q^{\sigma }\right) ^{-1}Z^{\sigma },  \label{eq:L}
\end{equation}%
where $Q^{\sigma}$ and $Z^{\sigma}$ are stable operators satisfying%
\begin{equation}
\sup_{\sigma \in \Xi }\left\{ \Lambda \bar{A}+\left[ 
\begin{array}{cc}
Z^{\sigma } & Q^{\sigma }%
\end{array}%
\right] \left[ 
\begin{array}{c}
\bar{C}^{\sigma } \\ 
\Lambda \bar{A}-I%
\end{array}%
\right] \right\} =0.  \label{eq:cond1}
\end{equation}%
Conversely, any generalized Luenberger observer (\ref{eq:GLO}) with observer
operator-gain $L$ in (\ref{eq:L}) is a stable state-estimator if%
\begin{equation}
\sup_{\sigma \in \Xi }\left\Vert \Lambda \bar{A}+\left[ 
\begin{array}{cc}
Z^{\sigma } & Q^{\sigma }%
\end{array}%
\right] \left[ 
\begin{array}{c}
\bar{C}^{\sigma } \\ 
\Lambda \bar{A}-I%
\end{array}%
\right] \right\Vert <1.  \label{eq:cond2}
\end{equation}
\end{theorem}

\begin{proof}
Suppose 
\begin{equation*}
\hat{x}=T^{\sigma }y_{a}^{\sigma },
\end{equation*}%
is a stable state-estimator for all $\sigma \in \Xi $. By Lemma \ref{lem:01}%
, $T^{\sigma }$ must be bounded and there exists a bounded operator $%
X^{\sigma }$ such that (\ref{eq:01}) holds. Now, define $Z^{\sigma }$ and $%
Q^{\sigma }$ as follows:%
\begin{eqnarray*}
Z^{\sigma } &:&=-T^{\sigma }, \\
Q^{\sigma } &=&-I-X^{\sigma }.
\end{eqnarray*}%
Then, direct calculation verifies%
\begin{eqnarray*}
&&\Lambda \bar{A}+\left[ 
\begin{array}{cc}
Z^{\sigma } & Q^{\sigma }%
\end{array}%
\right] \left[ 
\begin{array}{c}
\bar{C}^{\sigma } \\ 
\Lambda \bar{A}-I%
\end{array}%
\right] \\
&=&\Lambda \bar{A}+\left[ 
\begin{array}{cc}
-T^{\sigma } & -I-X^{\sigma }%
\end{array}%
\right] \left[ 
\begin{array}{c}
\bar{C}^{\sigma } \\ 
\Lambda \bar{A}-I%
\end{array}%
\right] \\
&=&\Lambda \bar{A}-T\bar{C}^{\sigma }+I-\Lambda \bar{A}-X^{\sigma }\Lambda 
\bar{A}+X^{\sigma } \\
&=&-T\bar{C}^{\sigma }+I+X^{\sigma }\left( I-\Lambda \bar{A}\right) \\
&=&-\left[ 
\begin{array}{cc}
T^{\sigma } & X^{\sigma }%
\end{array}%
\right] \left[ 
\begin{array}{c}
\bar{C}^{\sigma } \\ 
\Lambda \bar{A}-I%
\end{array}%
\right] +I \\
&=&0,
\end{eqnarray*}%
which implies (\ref{eq:cond1}) is satisfied. Therefore, any stable
state-estimator can be written as a generalized Luenberger observer. It
remains to show the converse. That is, any generalized Luenberger observer (%
\ref{eq:GLO}) with observer operator-gain (\ref{eq:L}) and (\ref{eq:cond2})
results in a stable state-estimator. Given a generalized Luenberger observer
(\ref{eq:GLO}), its estimation error is given by%
\begin{eqnarray}
e &=&\hat{x}-x=\Lambda \bar{A}\hat{x}+L^{\sigma }\left( \bar{C}^{\sigma }%
\hat{x}-y_{a}^{\sigma }\right) -\Lambda \bar{A}x-\Lambda \bar{B}w-\bar{x}_{0}
\notag \\
&=&\left( \Lambda \bar{A}+L^{\sigma }\bar{C}^{\sigma }\right) e-L^{\sigma }%
\bar{D}^{\sigma }w-\Lambda \bar{B}w-\bar{x}_{0}.  \label{eq:03}
\end{eqnarray}%
Assuming (\ref{eq:L})-(\ref{eq:cond1}), there exists a bounded operator $%
\mathcal{E}^{\sigma }$ with $\left\Vert \mathcal{E}^{\sigma }\right\Vert <1$
such that%
\begin{eqnarray}
\mathcal{E}^{\sigma } &\mathcal{=}&\Lambda \bar{A}+\left[ 
\begin{array}{cc}
Z^{\sigma } & Q^{\sigma }%
\end{array}%
\right] \left[ 
\begin{array}{c}
\bar{C}^{\sigma } \\ 
\Lambda \bar{A}-I%
\end{array}%
\right]  \label{eq:06} \\
&=&\left( I+Q^{\sigma }\right) \Lambda \bar{A}+Z\bar{C}^{\sigma }-Q^{\sigma }
\notag \\
&=&\left( I+Q^{\sigma }\right) \Lambda \bar{A}+\left( I+Q^{\sigma }\right)
L^{\sigma }\bar{C}^{\sigma }-Q^{\sigma }.  \notag
\end{eqnarray}%
Therefore,%
\begin{equation*}
\Lambda \bar{A}+L^{\sigma }\bar{C}^{\sigma }=\left( I+Q^{\sigma }\right)
^{-1}\left( Q^{\sigma }+\mathcal{E}^{\sigma }\right) .
\end{equation*}%
Using this expression in (\ref{eq:03}), we obtain%
\begin{eqnarray}
e &=&-\left\{ I-\left( \Lambda \bar{A}+L^{\sigma }\bar{C}^{\sigma }\right)
\right\} ^{-1}\left\{ L\bar{D}^{\sigma }w+\Lambda \bar{B}w+\bar{x}%
_{0}\right\}  \label{eq:07} \\
&=&-\left\{ I-\mathcal{E}^{\sigma }\right\} ^{-1}\left\{ Z^{\sigma }\bar{D}%
^{\sigma }w+\left( I+Q^{\sigma }\right) \Lambda \bar{B}w+\left( I+Q^{\sigma
}\right) \bar{x}_{0}\right\} .  \notag
\end{eqnarray}%
Notice that, since $\left\Vert \mathcal{E}^{\sigma }\right\Vert <1$, we have
that%
\begin{equation*}
\left\Vert \mathcal{E}^{\sigma }\left\{ I-\mathcal{E}^{\sigma }\right\}
^{-1}\right\Vert \leq \frac{\left\Vert \mathcal{E}^{\sigma }\right\Vert }{%
1-\left\Vert \mathcal{E}^{\sigma }\right\Vert },
\end{equation*}%
and hence the error, $e$, in the above expression is a bounded signal. In
fact,%
\begin{eqnarray*}
\left\Vert e\right\Vert &\leq &\frac{1}{1-\left\Vert \mathcal{E}\right\Vert }%
\times \\
&&\left\{ \left\Vert \left( I+Q^{\sigma }\right) \Lambda \bar{B}+Z^{\sigma }%
\bar{D}^{\sigma }\right\Vert \left\Vert w\right\Vert +\left\Vert \left(
I+Q^{\sigma }\right) \right\Vert \left\Vert \bar{x}_{0}\right\Vert \right\} .
\end{eqnarray*}%
This implies that any generalized Luenberger observer, with (\ref{eq:GLO})
and (\ref{eq:L})-(\ref{eq:cond1}), is a stable state-estimator and completes
the proof.
\end{proof}

\subsection{Optimal State-Estimator}

In this part, we present a resilient state-estimation design based on
Theorem \ref{thm:01}. We are interested to find the optimal state-estimator
such that the estimation error is minimized. According to Theorem \ref%
{thm:01}, any stable state-estimator can be written as%
\begin{equation*}
\hat{x}=\Lambda \bar{A}\hat{x}+L^{\sigma }\left( \bar{C}^{\sigma }\hat{x}%
-y_{a}^{\sigma }\right) ,
\end{equation*}%
where, given $\varepsilon \in \lbrack 0,1)$, there exist stable $Q^{\sigma }$
and $Z^{\sigma }$ such that 
\begin{eqnarray}
&&L^{\sigma }=\left( I+Q^{\sigma }\right) ^{-1}Z^{\sigma },  \notag \\
&&\sup_{\sigma \in \Xi }\left\Vert \Lambda \bar{A}+\left[ 
\begin{array}{cc}
Z^{\sigma } & Q^{\sigma }%
\end{array}%
\right] \left[ 
\begin{array}{c}
\bar{C}^{\sigma } \\ 
\Lambda \bar{A}-I%
\end{array}%
\right] \right\Vert <\varepsilon .  \label{eq:cond03}
\end{eqnarray}%
In above expression, the dependency of $L^{\sigma }$, $Q^{\sigma },$ and $%
Z^{\sigma }$ on the switching signal is made explicit. The underlying
assumption here is that the state-estimator knows the strategy of the
attacker causally. That is, at each given time $t$, the state-estimator has
the knowledge $\left\{ \sigma \left( 0\right) ,\sigma \left( 1\right)
,...,\sigma \left( t\right) \right\} $, but does not know the attacker's
strategy in future. We want to find $Q^{\sigma }$ and $Z^{\sigma }$ such
that while (\ref{eq:cond03}) is satisfied the estimation error is minimized.
The error estimation is derived in the proof of Theorem \ref{thm:01} given by%
\begin{eqnarray}
&&e=  \notag \\
&&-\left\{ I-\mathcal{E}^{\sigma }\right\} ^{-1}\left\{ Z^{\sigma }\bar{D}%
^{\sigma }w+\left( I+Q^{\sigma }\right) \Lambda \bar{B}w+\left( I+Q^{\sigma
}\right) \bar{x}_{0}\right\} , \\
&=&-\left\{ I-\mathcal{E}^{\sigma }\right\} ^{-1}\times  \label{eq:08} \\
&&\left\{ \left[ 
\begin{array}{cc}
\Lambda \bar{B} & I%
\end{array}%
\right] +\left[ 
\begin{array}{cc}
Z^{\sigma } & Q^{\sigma }%
\end{array}%
\right] \left[ 
\begin{array}{cc}
\bar{D}^{\sigma } & 0 \\ 
\Lambda \bar{B} & I%
\end{array}%
\right] \right\} \left[ 
\begin{array}{c}
w \\ 
\bar{x}_{0}%
\end{array}%
\right] .
\end{eqnarray}%
where%
\begin{equation*}
\mathcal{E}^{\sigma }\mathcal{=}\Lambda \bar{A}+\left[ 
\begin{array}{cc}
Z^{\sigma } & Q^{\sigma }%
\end{array}%
\right] \left[ 
\begin{array}{c}
\bar{C}^{\sigma } \\ 
\Lambda \bar{A}-I%
\end{array}%
\right] .
\end{equation*}

\begin{theorem}
There exists a stable state-estimator such that the induced norm from $\left[
\begin{array}{c}
w \\ 
\bar{x}_{0}%
\end{array}%
\right] $ to the estimation error $e$ is less than some positive real number 
$\gamma $ if and only if there exists stable operators $Q^{\sigma }$ and $%
Z^{\sigma }$ such that%
\begin{eqnarray}
\left\Vert \left[ 
\begin{array}{cc}
\Lambda \bar{B} & I%
\end{array}%
\right] +\left[ 
\begin{array}{cc}
Z^{\sigma } & Q^{\sigma }%
\end{array}%
\right] \left[ 
\begin{array}{cc}
\bar{D}^{\sigma } & 0 \\ 
\Lambda \bar{B} & I%
\end{array}%
\right] \right\Vert &\leq &\gamma ,  \label{eq:10} \\
\Lambda \bar{A}+\left[ 
\begin{array}{cc}
Z^{\sigma } & Q^{\sigma }%
\end{array}%
\right] \left[ 
\begin{array}{c}
\bar{C}^{\sigma } \\ 
\Lambda \bar{A}-I%
\end{array}%
\right] &=&0,  \label{eq:09}
\end{eqnarray}%
for any $\sigma \in \Xi $. In this case, the optimal cost is defined by%
\begin{equation}
\gamma ^{\ast }=\inf_{\left( \gamma ,Q^{\sigma },Z^{\sigma }\right)
}\sup_{\sigma \in \Xi }\gamma ,  \label{eq:20}
\end{equation}%
subject to (\ref{eq:10})-(\ref{eq:09}).
\end{theorem}

\begin{proof}
From Theorem \ref{thm:01}, the set of all stable state-estimator is
parametrized by $\left( Q^{\sigma },Z^{\sigma }\right) $ such that (\ref%
{eq:L}) and (\ref{eq:cond1}) hold. We notice that (\ref{eq:cond1}) is the
same as (\ref{eq:09}) and , from (\ref{eq:07}), the induced norm from $\left[
\begin{array}{c}
w \\ 
\bar{x}_{0}%
\end{array}%
\right] $ to $e$, when (\ref{eq:09}) is satisfied is given by 
\begin{equation*}
\left\Vert \left[ 
\begin{array}{cc}
\Lambda \bar{B} & I%
\end{array}%
\right] +\left[ 
\begin{array}{cc}
Z^{\sigma } & Q^{\sigma }%
\end{array}%
\right] \left[ 
\begin{array}{cc}
\bar{D}^{\sigma } & 0 \\ 
\Lambda \bar{B} & I%
\end{array}%
\right] \right\Vert .
\end{equation*}%
Therefore, the induced norm from $\left[ 
\begin{array}{c}
w \\ 
\bar{x}_{0}%
\end{array}%
\right] $ to the estimation error $e$ is less than $\gamma $ if (\ref{eq:10}%
) holds.
\end{proof}

We note that searching over stable systems $Q^{\sigma }$ and $Z^{\sigma }$
such that (\ref{eq:10})-(\ref{eq:09}) hold is a convex optimization but
infinite dimensional optimization. In what follows, we will reduce (\ref%
{eq:10})-(\ref{eq:09}) to finite dimensional convex optimization at the cost
of finding sub-optimal (but arbitrarily close to optimal) solutions. To this
end, according to Lemma \ref{lem:02}, since $\left( Q^{\sigma },Z^{\sigma
}\right) $ is stable, one can approximate them by FIR input-output switching
systems. In doing so, in general, it becomes challenging to satisfy (\ref%
{eq:09}) exactly and hence we need to relax (\ref{eq:09}). The result is
summarized in the following:

\begin{theorem}
\label{thm:02}Suppose there exist $0\leq \bar{\varepsilon}<1$ and FIR
input-output switching systems $Q^{\sigma },Z^{\sigma }\in \mathcal{S}%
_{IO}^{M}$ of some degree $M$ such that%
\begin{eqnarray}
\left\Vert \left[ 
\begin{array}{cc}
\Lambda \bar{B} & I%
\end{array}%
\right] +\left[ 
\begin{array}{cc}
Z^{\sigma } & Q^{\sigma }%
\end{array}%
\right] \left[ 
\begin{array}{cc}
\bar{D}^{\sigma } & 0 \\ 
\Lambda \bar{B} & I%
\end{array}%
\right] \right\Vert &\leq &\bar{\gamma},  \label{eq:21} \\
\left\Vert \Lambda \bar{A}+\left[ 
\begin{array}{cc}
Z^{\sigma } & Q^{\sigma }%
\end{array}%
\right] \left[ 
\begin{array}{c}
\bar{C}^{\sigma } \\ 
\Lambda \bar{A}-I%
\end{array}%
\right] \right\Vert &<&\bar{\varepsilon}.  \label{eq:22}
\end{eqnarray}%
Then the optimal cost $\gamma ^{\ast }$ satisfies%
\begin{equation*}
\gamma ^{\ast }\leq \bar{\gamma}+\frac{\bar{\varepsilon}}{1-\bar{\varepsilon}%
}\bar{\gamma}.
\end{equation*}
\end{theorem}

\begin{proof}
Immediate from the error dynamics given in (\ref{eq:08}).
\end{proof}

We emphasize that (\ref{eq:21})-(\ref{eq:22}) are in the so-called
model-matching form and they can be solved using the methods developed in 
\cite{naghnaeian2018l_p} with arbitrary accuracy.

\section{Illustrative Example}

In this section, we derive the optimal state-estimator for problem outlined
in Example \ref{exm:01}. We use Theorem \ref{thm:02} as basis of our
computations. The parameter values are given in Example \ref{exm:01} and the
attackers strategy can cause switches in the C- and D-matrices as given by (%
\ref{eq:C})-(\ref{eq:D}). First, we will find the optimal state-estimator
for the nominal case, i.e., when $\sigma \left( .\right) $ is constant and
identically equal to $0$. In this case, the optimal cost is $\gamma ^{\ast
}=5.0275$ and the state-estimator is given by 
\begin{equation*}
\hat{x}=Ty_{a}^{0},
\end{equation*}%
where the impulse response of $T=\left\{ T\left( k\right) \right\}
_{k=0}^{\infty }$ is given by%
\begin{eqnarray*}
T\left( 0\right)  &=&\left[ 
\begin{array}{cc}
0 & 0.75 \\ 
0.74 & -0.065 \\ 
-0.126 & -0.031%
\end{array}%
\right] , \\
T\left( 1\right)  &=&\left[ 
\begin{array}{cc}
-1.25 & -2 \\ 
0.195 & 0 \\ 
-0.906 & -1%
\end{array}%
\right] , \\
T\left( k\right)  &=&0,\text{ for }k\geq 2\text{.}
\end{eqnarray*}%
This state-estimator, however, does not result in a stable approximation
error in the presence of DoS attack. One can use the method developed in
this paper to find a stable state-estimator that is resilient to DoS attack
strategy. In this example, we apply Theorem \ref{thm:02} and search for
input-output switching $\left( Q^{\sigma },Z^{\sigma }\right) $ of degree $1$%
. Furthermore, we let $\Xi $ to be the set of all binary sequences. For this
case, we manage to find a stable state-estimator with optimal cost of $%
\gamma ^{\ast }=32.5$. The optimal state-estimator is given by $\hat{x}%
=T^{\sigma }y_{a}^{\sigma }$ where $T^{\sigma }$ is an output-only switching
system of degree one. At each time instant $t$, 
\begin{equation*}
\hat{x}\left( t\right) =\sum_{\tau =t-4}^{t}T^{\sigma \left( t\right)
}\left( t-\tau \right) y_{a}^{\sigma \left( \tau \right) }\left( \tau
\right) ,
\end{equation*}%
where%
\begin{equation*}
T^{i}=\left\{ T^{i}\left( k\right) \right\} _{k=0}^{4},\text{ for }i=1,2%
\text{,}
\end{equation*}%
is the FIR impulse response of $T^{i}$. The numerical values for the impulse
response terms of $T^{1}$ are%
\begin{eqnarray*}
T^{1}\left( 0\right)  &=&\left[ 
\begin{array}{cc}
0.24 & 0.44 \\ 
0.37 & -0.17 \\ 
0.06 & -0.17%
\end{array}%
\right] ,T^{1}\left( 1\right) =\left[ 
\begin{array}{cc}
0.25 & 0 \\ 
0.39 & 0 \\ 
0.94 & 0%
\end{array}%
\right] , \\
T^{1}\left( 2\right)  &=&\left[ 
\begin{array}{cc}
0.02 & 0 \\ 
0.14 & 0 \\ 
0.07 & 0%
\end{array}%
\right] ,T^{1}\left( 3\right) =\left[ 
\begin{array}{cc}
-0.21 & 0 \\ 
0.01 & 0 \\ 
-0.3 & 0%
\end{array}%
\right] , \\
T^{1}\left( 4\right)  &=&\left[ 
\begin{array}{cc}
-2.1 & 0 \\ 
-0.06 & 0 \\ 
-0.93 & 0%
\end{array}%
\right] 
\end{eqnarray*}%
And the impulse reponse of $T^{2}$ is given by%
\begin{eqnarray*}
T^{2}\left( 0\right)  &=&\left[ 
\begin{array}{cc}
-1.5 & 0 \\ 
0.98 & 0 \\ 
0.66 & 0%
\end{array}%
\right] ,T^{2}\left( 1\right) =\left[ 
\begin{array}{cc}
3.75 & 0 \\ 
0.07 & 0 \\ 
0.61 & 0%
\end{array}%
\right] , \\
T^{2}\left( 2\right)  &=&\left[ 
\begin{array}{cc}
0 & 0 \\ 
-0.06 & 0 \\ 
-0.05 & 0%
\end{array}%
\right] ,T^{2}\left( 3\right) =\left[ 
\begin{array}{cc}
0 & 0 \\ 
-0.03 & 0 \\ 
-0.45 & 0%
\end{array}%
\right] , \\
T^{2}\left( 4\right)  &=&\left[ 
\begin{array}{cc}
-2.25 & 0 \\ 
0.05 & 0 \\ 
-0.77 & 0%
\end{array}%
\right] .
\end{eqnarray*}

\bigskip

\section{Conclusion}

In this paper, utilizing the operator framework, we first parametrized the
set of all stable state-estimators resilient to DoS attack. This was carried
out by converting the problem to a state estimation problem for linear
switched systems where the attacker's strategy prescribes the switching law.
Furthermore, we showed that the set of generalized Luenberger observers
captures all stable state-estimators. Then, we cast the problem of finding
the optimal estimator as a convex optimization over the set of stable
factors of the observer operator-gain. This optimization, for the $l_{\infty}
$ induced norm, can be rewritten as a linear program which can be solved
efficiently.

\bibliographystyle{IEEEtran}
\bibliography{BIB_CDC2019}

\end{document}